\documentclass[12pt]{amsart}

\usepackage{amsmath,amssymb,amsthm,array,microtype,csvsimple,longtable,multicol,caption,multirow,tikz,float}
\usetikzlibrary{patterns}
\usepackage[cal=rsfso]{mathalfa}
\usepackage[margin=2.75cm]{geometry}

\newtheorem{theorem}{Theorem}[section]

\newtheorem{corollary}[theorem]{Corollary}
\newtheorem{lemma}[theorem]{Lemma}
\newtheorem{conjecture}[theorem]{Conjecture}

\theoremstyle{definition}
\newtheorem{definition}[theorem]{Definition}
\newtheorem{example}[theorem]{Example}

\theoremstyle{remark}

\newcommand{\p}{\ \ \circle*{7}}
\newcommand{\emp}{\ \ \circle{7}}
\newcommand{\ov}{-}

\newcolumntype{C}[1]{>{\centering\arraybackslash}p{#1}}

\newcommand{\R}{\mathbb{R}}

\begin{document}



\title[Association and Simpson Conversion]{Association and Simpson Conversion \\ in $\mathbf{2 \times 2 \times 2}$ Contingency Tables}

\author{Svante Linusson}
\address{KTH Royal Institute of Technology, Department of Mathematics, SE-100 44, Stockholm, Sweden}
\email{linusson@math.kth.se}

\author{Matthew T. Stamps}
\address{Yale-NUS College, Division of Science, 10 College Avenue West \#01-101, Singapore, 138609}
\email{matt.stamps@yale-nus.edu.sg}


\date{\today}  

\begin{abstract}
We study a generalisation of Simpson reversal (also known as Simpson's paradox or the Yule-Simpson effect) to $2 \times 2 \times 2$ contingency tables and characterise the cases for which it can and cannot occur with two combinatorial-geometric lemmas.  We also present a conjecture based on some computational experiments on the expected likelihood of such events.  
\end{abstract}

\maketitle


\section{Introduction}

Simpson reversal is a phenomenon in statistics in which a common trend among different groups disappears, or even reverses, when the groups are combined.  More precisely, if $A_1$, $A_2$, and $B$ are events such that $A_1$ and $A_2$ are positively associated given $B$ and also given its complement $\overline{B}$, i.e., $$P(A_1 \cap A_2 \, | \, B) > P(A_1 \, | \, B) P(A_2 \, | \, B) \quad \text{and} \quad P(A_1 \cap A_2 \, | \, \overline{B}) > P(A_1 \, | \, \overline{B}) P(A_2 \, | \, \overline{B}),$$ then it is possible for $A_1$ and $A_2$ to be independent to each other or negatively associated marginally, independent of $B$, i.e., $$P(A_1 \cap A_2) \leq P(A_1) P(A_2).$$   Real-world examples of the paradox are well-documented, for instance in \cite{Bickel-Hammel-O'Connell} and \cite{Ross}, but it has also been observed and studied in several applications from biology, such as \cite{Julious-Mullee} and \cite{Wilcox}. 

In this paper, we consider a generalisation of Simpson reversal to trios of events with respect to a fourth.  The practicality of such a generalisation can be illustrated with the following scenario that is motivated by a recent paper in biology \cite{Eble}.  When measuring the association between two bacterial DNA loci in a sample, it is important to be aware of Simpson reversal since it is possible that the measured association might actually be the opposite if measured on two subsets of the bacteria separately, which could lead to a misinterpretation of the data. In the same manner, it is important to understand the possible misinterpretations that could occur when measuring association on three loci, sometimes called the \emph{fitness landscape} \cite{BPS}. This paper is a first step in understanding the statistical pitfalls in such investigations.

The generalisation of Simpson reversal to trios is significantly more complex than Simpson reversal for pairs since a trio of random events can satisfy a combination of mutual, marginal, and conditional associations with respect to one another whereas a pair of events can be positively or negatively associated or independent.  We ask if $A_1$, $A_2$, $A_3$, and $B$ are events such that $A_1$, $A_2$, and $A_3$ satisfy a common set of mutual, marginal, and conditional associations given $B$ and also given its complement, what sets of associations can the $A_1$, $A_2$, and $A_3$ satisfy independent of $B$?  We show that while there are many ways in which a set of associations among $A_1$, $A_2$, and $A_3$ can change (we call such instances \emph{Simpson conversions}), it is not possible for every set of associations to convert into every other set of associations.  Our main result, Theorem~\ref{thm:main}, characterises the Simpson conversions for trios of  events.  The proof extends a well-known geometric interpretation of Simpson reversal in terms of triangulations of the square to a geometric interpretation of Simpson conversion to triangulations of the cube. Our characterisation consists of two parts:  First, we establish some combinatorial-geometric lemmas (involving triangulations of the cube) to preclude certain instances of Simpson conversions.  Then, we verify experimentally instances of all the other Simpson conversions.  

The remainder of this paper is structured as follows:  In Section~\ref{S:2D}, we review a well-known example of Simpson reversal and several geometric interpretations in the literature.  In Section~\ref{S:3D}, we propose the generalisation of Simpson reversal for $2\times2\times2$ contingency tables, which we call Simpson conversion, that involves an arrangement of hyperplanes in $\mathbb{R}^8$ and the set of triangulations of the $3$-dimensional cube.  We describe the relationship between the linear forms defining said hyperplane arrangement and the triangulations in Section~\ref{S:corr}.  Section~\ref{S:switch} contains the main results of the paper, in which we characterise the $3$-dimensional analog of Simpson's reversals for $2\times2\times2$ contingency tables.  We conclude the paper with a conjecture on the frequency of Simpson conversion and some observations based on computational experiments in Section~\ref{S:amalga}. 

\section{Simpson Reversal in Two Dimensions}\label{S:2D}

Here we review a well-known example and several geometric interpretations of Simpson reversal in two dimensions.

\begin{example}\label{DP-data}
A concrete example of Simpson reversal is illustrated by the voting results for (the Senate version) of the Civil Rights Act of 1964 in the United States House of Representatives.  The votes are listed below broken down according to political party (Democrats and Republicans) and region of the country (Northern, Southern, and all states), as presented in \cite{guardian}. 
\begin{table}[ht]
\caption{House of Representatives voting results for the Civil Rights Act of 1964 according to political party among Northern states (left), Southern states (middle), or all states (right).}
\centering
\begin{tabular}{c|cc}
Northern & Yes & No \\
\hline
Democrats & 144 & 8 \\
Republicans &137 & 24 \\ 
\end{tabular}
\quad
\begin{tabular}{c|cc}
Southern & Yes & No \\
\hline
Democrats & 8 & 83 \\
Republicans & 0 & 11 \\ 
\end{tabular}
\quad
\begin{tabular}{c|cc}
All & Yes & No \\
\hline
Democrats & 152 & 91 \\
Republicans & 137 & 35 \\ 
\end{tabular}
\label{tab:2D}
\end{table}
From these tables, one can observe that a higher percentage of Democrats voted in favour of the bill in both the Northern and the Southern states (95\% and 9\% of the Democrats compared to 85\% and 0\% of the Republicans, respectively), but a higher percentage of Republicans voted in favour of the bill overall (80\% of the Republicans compared to 63\% of the Democrats).  This instance of Simpson reversal can be explained by the fact that the relationship between party and vote on the bill was significantly affected by the regions corresponding to the voters.  Indeed, most of the Southern representatives at that time were Democrats, and the vast majority of negative votes came from that region.
\end{example}

While the existence of a Simpson reversal is surprising -- even paradoxical -- at first glance, there are several geometric interpretations that illustrate why the phenomenon is not only possible, but a relatively common occurrence.  A well known geometric interpretation of Simpson reversal is as follows:  Suppose $v_1$, $v_2$, $w_1$, and $w_2$ are vectors in $\R_{+}^2$ based at the origin such that the slope of $v_i$ is greater than the slope of $w_i$ for $i = 1, 2$.  Then it is possible that the slope of $v_1 + v_2$ is less than the slope of $w_1 + w_2$ as shown in Figure~\ref{F:vectors}.  In Example \ref{DP-data} it corresponds to $v_1=(8,144), v_2=(83,8), w_1=(24,137), w_2=(11,0)$.

\begin{figure}[ht]
\centering
\begin{tikzpicture}[scale=0.8]
\draw [-,dashed,red] (3,1)--(8,7);
\draw [-,dashed,red] (5,6)--(8,7);
\draw [-,dashed,blue] (2,3)--(9,6);
\draw [-,dashed,blue] (7,3)--(9,6);
\draw [->,thick,red] (0,0)--(3,1);
\draw [->,thick,red] (0,0)--(5,6);
\draw [->,thick,blue] (0,0)--(2,3);
\draw [->,thick,blue] (0,0)--(7,3);
\draw [->,ultra thick,red] (0,0)--(8,7);
\draw [->,ultra thick,blue] (0,0)--(9,6);
\draw [->,ultra thick] (0,0)--(11,0);
\draw [->,ultra thick] (0,0)--(0,8);
\node [above left,blue] at (2,3) {$v_1$};
\node [below right,blue] at (7,3) {$v_2$};
\node [above left,red] at (5,6) {$w_1$};
\node [below right,red] at (3,1) {$w_2$};
\node [above right,blue] at (9,6) {$v_1 + v_2$};
\node [above right,red] at (8,7) {$w_1 + w_2$};
\end{tikzpicture}
\caption{A geometric illustration of Simpson reversal.}
\label{F:vectors}
\end{figure}

Another interpretation of Simpson reversal can be seen geometrically via triangulations of a square.  A real-valued function $f: P_0 \to \R$ on the vertices $P_0$ of a convex polygon $P$ induces a unique triangulation on $P$ by taking the convex hull of the set $$P_f = \{(p_1,p_2,f(p_1,p_2)) \ | \ (p_1,p_2) \in P_0\} \subseteq \mathbb{R}^3$$ and projecting its upper envelope onto $P$, provided no four points in $P_f$ are coplanar.  This process is nicely explained in \cite{DeLoera-Rambau-Santos} and illustrated for the case where $P$ is a square in Figure~\ref{F:d=2}.  

\begin{figure}[ht]
\centering
\includegraphics[scale=0.9]{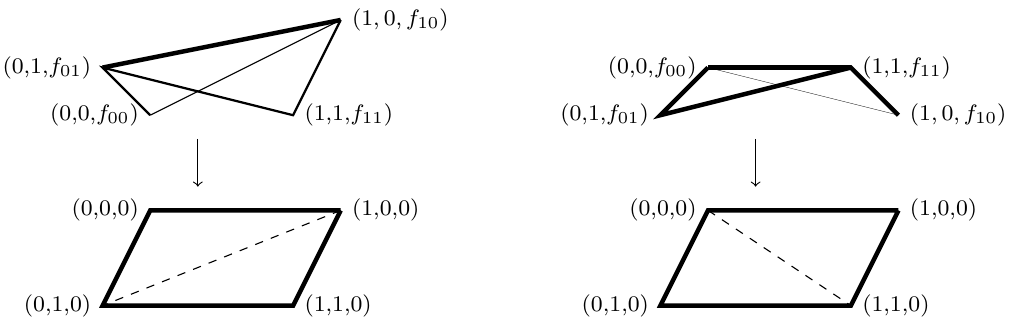}
\caption{Positive association (right) and negative association (left) induce different triangulations of the underlying square.}
\label{F:d=2}
\end{figure}

Note that a function $f : \{0,1\}^2 \to \mathbb{R}$ induces the triangulation of the square with the diagonal edge between the points $(0,0)$ and $(1,1)$ if $f(0,0)+f(1,1) > f(0,1)+f(1,0)$ and between the points $(0,1)$ and $(1,0)$ if the inequality is reversed.  Since a $2\times 2$ contingency table with values $F_{00}, F_{10},F_{01},F_{F11}$ is positively associated if $F_{00}\cdot F_{11}>F_{01}\cdot F_{10}$ and negatively associated if the inequality is reversed, the sign of the association of the contingency table is encoded by the triangulation of the square induced by the function $f : \{0,1\}^2 \to \mathbb{R}$ given by $f(x,y) = \ln (F_{xy})$.   Using upper case letters to denote the entries of a table and lower case letters to denote the corresponding function (as defined above), a pair of contingency tables $F$ and $G$ exhibit a Simpson reversal when their corresponding functions $f$ and $g$ induce the same triangulation of the square while the function $h : \{0,1\}^2 \to \mathbb{R}$ given by $h(x,y) = \ln (F_{xy}+G_{xy})$ corresponding to the table $F+G$ induces the other triangulation of the square. 

To see how such an occurrence is possible, observe that while the association of a contingency table $F$ is determined by the linear form $w := f_{00} + f_{11} - f_{01} - f_{10}$ where $f_{xy} = \ln F_{xy}$ that decomposes $\mathbb{R}^4$ into convex cones (half-spaces), the corresponding regions in the space of upper case letters, where the addition of contingency tables is done, are non-convex. Thus, the sum of two contingency tables belonging to the region of one association can belong to the other association.  It is interesting to consider what these non-convex regions look like, but that is not the focus of this paper.

\begin{example}
To illustrate the ideas above with the data in Example~\ref{DP-data}, let $N$, $S$, and $A = N+S$ denote the contingency tables in Table~\ref{tab:2D} corresponding to Northerm, Southern, and all states, respectively, and label the upper left, upper right, lower left, and lower right entires of each table by $(0,0)$, $(1,0)$, $(0,1)$, and $(1,1)$, respectively.  We can see that $N$ and $S$ are both positively associated while their sum $A$ is negatively associated since \begin{align*} N_{00} \cdot N_{11} = 144 \cdot 24 &> 137 \cdot 8 = N_{01} \cdot N_{10}, \\ S_{00} \cdot S_{11} = 8 \cdot 11 &> 83 \cdot 0 = S_{01} \cdot S_{10}, \text{ and} \\ A_{00} \cdot A_{11} = 152 \cdot 35 &< 91 \cdot 137 = N_{01} \cdot N_{10}. \end{align*} 
\end{example}

With these interpretations, the existence of Simpson reversals is not so surprising.  In fact, from the perspective of causality, see \cite{Pearl} or \cite{Pearl-Mackenzie}, it is a rather natural notion.  

\section{Associations in Three Dimensions}\label{S:3D}

The instance of Simpson reversal in the House of Representatives vote on the Civil Rights Bill of 1964 presented in Example~\ref{DP-data} (taken from \cite{guardian}) can also be observed in the Senate vote on the same bill, the results of which are listed below appended to the data in Table~\ref{tab:2D}.

\begin{table}[ht]
\caption{Voting results for the Civil Rights Act of 1964 according to legislative chamber and political party among Northern states (left), Southern states (middle), or all states (right).}
\centering
\begin{tabular}{c|c|cc}
\multicolumn{2}{c|}{Northern States} & Yes & No \\
\hline
 \multirow{2}{3em}{House} & Democrats & 144 & 8 \\
 & Republicans &137 & 24 \\ 
 \hline
\multirow{2}{3em}{Senate} & Democrats & 45 & 1 \\
& Republicans & 27 & 5 \\ 
\end{tabular}
\qquad
\begin{tabular}{c|c|cc}
\multicolumn{2}{c|}{Southern States} & Yes & No \\
\hline
 \multirow{2}{3em}{House} & Democrats & 8 & 83 \\
 & Republicans & 0 & 11 \\ 
 \hline
\multirow{2}{3em}{Senate} & Democrats & 1 & 20 \\
& Republicans & 0 & 1 \\ 
\end{tabular}

\bigskip
\begin{tabular}{c|c|cc}
\multicolumn{2}{c|}{All States} & Yes & No \\
\hline
 \multirow{2}{3em}{House} & Democrats & 152 & 91 \\
 & Republicans &137 & 35 \\ 
 \hline
\multirow{2}{3em}{Senate} & Democrats & 46 & 21 \\
& Republicans & 27 & 6 \\ 
\end{tabular}
\label{all_data}
\end{table}

With this, one might be interested in studying the various relationships between party, vote on the bill, and chamber of congress simultaneously by combining the two $2 \times 2$ contingency tables corresponding to each chamber of congress over each region into a $2 \times 2 \times 2$ contingency table.  Understanding how observed associations change when the data from different regions are combined becomes much more complicated than in the 2-dimensional case since there are many more notions of association for a $3$-way contingency table.  Indeed, a $2 \times 2 \times 2$ contingency table can exhibit mutual (all variables dependent on each other), marginal (two variables are dependent ignoring the third), and conditional (two variables are dependent given the third) associations, and these three types of association are distinct from one another in the sense that dependencies of one type do not necessarily imply dependencies of the other types.  More precisely, if $A_1$, $A_2$, and $A_3$ are random events, then there are eight distinct relations arising from mutual dependencies of the form $$P(X \cap Y \cap Z) \neq P(X)P(Y)P(Z);$$ three distinct relations arising from marginal dependencies of the form $$P(X \cap Y) \neq P(X)P(Y);$$ and six distinct relations arising from conditional dependencies of the form $$P(X \cap Y \, | \, Z) \neq P(X \, | \, Z)P(Y \, | \, Z)$$ where $X$, $Y$, and $Z$ are distinct events chosen from $A_1$ or $\overline{A_1}$, $A_2$ or $\overline{A_2}$, and $A_3$ or $\overline{A_3}$.  Note that there are only three and six distinct relations from marginal and conditional dependencies, respectively, since $$P(X \cap Y) \neq P(X)P(Y) \quad \Longleftrightarrow \quad P(X \cap \overline{Y}) \neq P(X)P(\overline{Y}).$$  

\begin{example}
For the data in Table~\ref{all_data}, voting ``Yes'' (Y) and being a Democrat (D) are positively marginally associated, independent of the legislative chamber, for both Northern (N) and Southern (S) states, but negatively marginally associated for all states, since \begin{align*} P(\text{Y} \cap \text{D} \, | \, \text{N}) = \frac{189}{391} &> \frac{353}{391} \cdot \frac{198}{391}  = P(\text{Y} \, | \, \text{N}) \cdot P(\text{D} \, | \, \text{N}), \\[1em] P(\text{Y} \cap \text{D} \, | \, \text{S}) = \frac{9}{124} &> \frac{9}{124} \cdot \frac{112}{124} = P(\text{Y} \, | \, \text{S}) \cdot P(\text{D} \, | \, \text{S}), \text{and} \\[1em] P(\text{Y} \cap \text{D}) = \frac{198}{515} &< \frac{362}{515} \cdot \frac{310}{515} = P(\text{Y}) \cdot P(\text{D}).\end{align*}  Therefore, the marginal association between voting ``Yes'' and being a Democrat, independent of the legislative chamber, in Table~\ref{all_data} exhibits a Simpson reversal.  We leave it to the reader to check that all of the other marginal and mutual, but only two of the conditional, associations in Table~\ref{all_data} exhibit Simpson reversals.
\end{example}

For a $2\times 2\times 2$ contingency table with values $F_{xyz}$, these 17 dependencies give linear relations on the variables $f_{xyz}=\ln F_{xyz}$, which have a natural correspondence with the set of linear relations arising from the 74 triangulations of the 3-dimensional cube.  Just as a generic $2 \times 2$ contingency table induces a triangulation of the square, a generic $2 \times 2 \times 2$ contingency table induces a triangulation of the cube (into tetrahedra) via a projection of the upper envelope of the convex hull of the points $(x,y,z,f_{xyz})$ in $\mathbb R^4$.  \cite{BPS} showed that the contingency tables that induce each triangulation of the cube are determined by 20 linear relations, which we list in Appendix~A.  The conditional associations correspond to the forms labeled $a$ through $f$; the marginal associations correspond to the sums $g+h$, $i+j$, and $k+l$; and the mutual associations correspond to the forms labeled $m$ through $t$.  The triangulations of the cube are listed in Appendix B, organised into six types according to symmetry -- every triangulation of a given type can be obtained from any other triangulation of that type via a rotation or reflection.  Triangulations of Type I consist of five tetrahedra while all other triangulations consist of six.  We have used the same notation as in \cite{BPS} for easy comparison, and we give a detailed explanation of how to interpret the figures in Appendix~B in Example~\ref{example:cube-notation}. 

\section{The Correspondence Between Linear Forms and Triangulations}\label{S:corr}

In this section, we describe the algebro-geometric correspondence between the linear forms in Appendix~A and the triangulations of the 3-dimensional cube in Appendix~B.  The positive quadrant of $\mathbb R^8$ is divided into 74 regions by the 20 hyperplanes associated to the linear forms, where each region corresponds to a unique triangulation of the cube.  The signs of the forms a-f correspond to the face diagonals on the six squares that make up the surface of the cube; the signs of the forms g-l correspond to flipping the interior diagonal of the cube within the six rectangles formed by opposite pairs of edges in the cube (these six rectangles each slice the cube into two triangular prisms); and finally the forms m-t correspond to whether or not the interior diagonal is present when passing between the triangulations of Types I and II. 

\newpage

\begin{example}\label{example:cube-notation}
Consider the triangulation labeled 3 in Appendix B shown below:
\begin{multicols}{2}
This triangulation consists of six tetrahedra, namely $\{000,001,010,100\}$, $\{011,001,010,100\}$, $\{011,001,111,100\}$, $\{101,001,111,100\}$, $\{011,010,111,100\}$, and $\{110,010,111,100\}$, where the notation $xyz$ means $(x,y,z)$.  \smallskip
\begin{center}
\begin{tikzpicture}[scale=0.9]
\put(0,0) 
{\node(a1) at (5,2){\p};
\node(a2) at (5,6){\emp};
\node(a3) at (1,6){\p};
\node(a4) at (1,2){\emp};
\node(a5) at (0,0){\p};
\node(a6) at (4,0){\emp};
\node(a7) at (4,4){\p};
\node(a8) at (0,4){\emp};
\node [right] at (a1) {\tiny 100};
\node [right] at (a2) {\tiny 101};
\node [above] at (a3) {\tiny 001};
\node [above right] at (a4) {\tiny 000};
\node [left] at (a5) {\tiny 010};
\node [right] at (a6) {\tiny 110};
\node [above left] at (a7) {\tiny 111};
\node [left] at (a8) {\tiny 011};
\draw [-,ultra thick] (a1.center) --  (a2.center)--  (a3.center)--  (a4.center)--  (a5.center)--  (a6.center)--  (a7.center)--  (a8.center)--  (a5.center);
\draw [-,ultra thick] (a1.center) --  (a6.center);
\draw [-,ultra thick] (a2.center) --  (a7.center);
\draw [-,ultra thick] (a3.center) --  (a8.center);
\draw [-,ultra thick] (a1.center) --  (a4.center);
\draw [-, blue] (a1.center) -- (a7.center);
\draw [-, blue] (a1.center) -- (a3.center);
\draw [-, blue] (a1.center) -- (a5.center);
\draw [-, blue] (a3.center) -- (a7.center);
\draw [-, blue] (a5.center) -- (a7.center);
\draw [-, blue] (a3.center) -- (a5.center);
\draw [-, red] (a1.center) -- (a8.center);
\node at (2,-0.7) {$b,d,\ov e, \ov t$};
}
\end{tikzpicture}
\end{center}The sequence of letters $b,d,\ov e, \ov t$ below the cube indicates this triangulation is induced by the contingency tables $F$ satisfying the relations
\begin{align*}
0<b&:=f_{001}+f_{111}-f_{011}-f_{101}, \\
0<d&:=f_{010}+f_{111}-f_{110}-f_{011}, \\
0>e&:=f_{000}+f_{011}-f_{010}-f_{001}, \text{ and} \\
0>t&:=f_{010}+f_{001}+f_{111}-f_{100}-2f_{011},
\end{align*}
from Appendix A or, equivalently,
\begin{align*}
\frac{F_{001}F_{111}}{F_{011}F_{101}}&>1, \quad \frac{F_{010}F_{111}}{F_{110}F_{011}}>1, \\[0.5em]
\frac{F_{000}F_{011}}{F_{010}F_{001}}&<1, \quad \frac{F_{010}F_{001}F_{111}}{F_{100}F^2_{011}}<1. 
\end{align*}
The value on the other 16 linear forms follows from these 4. 
\end{multicols}
\end{example}

Notice that the vertices $001$, $100$, $010$, and $111$ in Triangulation 3 are each incident with three face (blue) diagonals, and the remaining four vertices are not incident with any face diagonals. Such vertices will play an important role in the next section, so we will give them a name.  

\begin{definition}  For a given triangulation of the cube, \begin{itemize} \item a vertex is called \emph{full} and marked with a filled circle in Example~\ref{example:cube-notation} and Appendix~B if it is incident with all three face diagonals belonging to the square faces that contain it; \item a vertex is called \emph{empty} and marked with an empty circle in Example~\ref{example:cube-notation} and Appendix~B if it is not incident with any of the three face diagonals belonging to the square faces that contain it. \end{itemize}
\end{definition}

One can see in Appendix B that triangulations of Types I and II have four full and four empty vertices; triangulations of Type III have two full and two empty vertices; triangulations of Type IV have no full and two empty vertices; triangulations of Type V have one full and one empty vertex; and triangulations of Type VI have two full and no empty vertices.

We conclude this section with some additional observations on the relationship between associations and triangulations of the cube:  In the triangulations of Type I and II, the face diagonals in each pair of opposite faces have opposite directions, which means that the association of any two variables in the corresponding contingency tables are dependent on the value of the third.  On the contrary, the face diagonals in each pair of opposite faces in the triangulations of Type IV and VI have the same direction, which means that the association of any two variables in the corresponding contingency tables is not dependent on the value of the third.  Triangulations of Type III have exactly one pair of opposite faces whose face diagonals have the same direction while triangulations of Type V have exactly one pair of opposite faces whose face diagonals have different directions. 

\section{Simpson Conversion in Three Dimensions}\label{S:switch}

This is the main section where we consider the question:  If a pair of $2 \times 2 \times 2$ contingency tables induce the same triangulation of the cube, is it possible that their component-wise sum induces a different triangulation?  Just as in the case of $2 \times 2$ contingency tables, it is not difficult to find an example that answers this question in the affirmative.  Unlike the case of $2 \times 2$ contingency tables, however, there are many different ways in which these instances can arise:  Instead of reversing from one (of two) triangulations to the other, it is possible to convert from one triangulation of a cube to several of the other 73.  We call a pair of $2 \times 2 \times 2$ contingency tables that induce the same triangulation $A$ of the cube whose sum induces a different triangulation $B$ of the cube a \emph{Simpson conversion from $A$ to $B$}.  Our main theorem is a characterisation of the pairs of triangulations, $A$ and $B$, for which there exists a Simpson conversion from $A$ to $B$. We proceed with some essential lemmas for the characterization.

\subsection{Setup and Essential Lemmas}

For each $x \in \{0,1\}$, we define $$\overline{x} = \begin{cases} 1 & x = 0, \\ 0 & x = 1. \end{cases}$$  We denote by $F_{v}$ the value of the  function $F : \{0,1\}^k \to \mathbb{R}_+$ on input $v \in \{0,1\}^k$, i.e., the entry in the corresponding contingency table.   We start with a lemma giving relations that follow if we have Simpson reversal happening in two dimensions.

\begin{lemma}\label{basic}
Let $F,G : \{0,1\}^2 \to \R_{+}$ and $(x,y) \in \{0,1\}^2$.  If $F_{\overline{x}y} F_{x\overline{y}} < F_{xy} F_{\overline{x}\overline{y}},$ $G_{\overline{x}y} G_{x \overline{y}} < G_{xy} G_{\overline{x}\overline{y}},$ and $$(F_{\overline{x}y}+G_{\overline{x}y}) (F_{x \overline{y}}+G_{x \overline{y}}) > (F_{xy}+G_{xy}) (F_{\overline{x}\overline{y}}+G_{\overline{x}\overline{y}}),$$ then one of the following pairs of inequalities must hold  \begin{multline*}F_{\overline{x}y} G_{xy}  >  F_{xy} G_{\overline{x}y} \ \text{ and } \ F_{x \overline{y}} G_{xy}  <  F_{xy} G_{x \overline{y}} \\ \text{or} \\ F_{\overline{x}y} G_{xy}  <  F_{xy} G_{\overline{x}y} \ \text{ and } \ F_{x \overline{y}} G_{xy} > F_{xy} G_{x \overline{y}}. \end{multline*}
\end{lemma}

\begin{proof}
It suffices to prove this for the case where $x = y = 0$.  Suppose $F_{10} F_{01} < F_{00} F_{11}$, $G_{10} G_{01} < G_{00} G_{11}$, and $(F_{10}+G_{10}) (F_{01}+G_{01}) > (F_{00}+G_{00}) (F_{11}+G_{11})$.  Expanding the products in the third inequality and multiplying each side by $F_{00}G_{00}$, we get that 
\begin{align*} & F_{00}F_{10}F_{01}G_{00} + F_{00}F_{10}G_{00}G_{01} + F_{00}F_{01}G_{00}G_{10} + F_{00}G_{00}G_{10}G_{01} \\ &> F_{00}F_{00}F_{11}G_{00} + F_{00}F_{00}G_{00}G_{11} + F_{00}F_{11}G_{00}G_{00} + F_{00}G_{00}G_{00}G_{11}.\end{align*}
Applying the first and second inequalities to the first and fourth summands, we get that $$F_{00}F_{10}G_{00}G_{01} + F_{00}F_{01}G_{00}G_{10} > F_{00}F_{00}G_{00}G_{11} + F_{00}F_{11}G_{00}G_{00},$$ and hence $(F_{10}G_{00} - F_{00}G_{10})(F_{00}G_{01}-F_{01}G_{00}) > 0.$  The desired result follows immediately.  
\end{proof}
 
The next lemma says that if two $2 \times 2 \times 2$ contingency tables, $F$ and $G$, both have a full vertex at $xyz$, then the sum $F+G$ cannot have an empty vertex at $xyz$.  That is, if all three face diagonals are incident to vertex $xyz$ in both $F$ and $G$, then they cannot all flip in the sum $F+G$.

\begin{lemma}\label{advanced}
Let $F,G : \{0,1\}^3 \to \R_{+}$ and $(x,y,z) \in \{0,1\}^3$ .  If each of the following six inequalities hold \begin{align*} 
F_{\overline{x}yz} F_{x\overline{y}z} < F_{xyz} F_{\overline{x}\overline{y}z}, 
&\quad 
G_{\overline{x}yz} G_{x\overline{y}z} < G_{xyz} G_{\overline{x}\overline{y}z},  \\ 
F_{\overline{x}yz} F_{xy\overline{z}} < F_{xyz} F_{\overline{x}y\overline{z}}, 
&\quad 
G_{\overline{x}yz} G_{xy\overline{z}} < G_{xyz} G_{\overline{x}y\overline{z}}, \\ 
F_{x\overline{y}z} F_{xy\overline{z}} < F_{xyz} F_{x\overline{y}\overline{z}}, 
&\quad
G_{x\overline{y}z} G_{xy\overline{z}} < G_{xyz} G_{x\overline{y}\overline{z}},  \end{align*} 
then it is not possible for all three of the following three inequalities to hold
\begin{align*} 
(F_{\overline{x}yz}+G_{\overline{x}yz}) (F_{x\overline{y}z}+G_{x\overline{y}z}) &> (F_{xyz}+G_{xyz}) (F_{\overline{x}\overline{y}z}+G_{\overline{x}\overline{y}z}), \\
(F_{\overline{x}yz}+G_{\overline{x}yz}) (F_{xy\overline{z}}+G_{xy\overline{z}}) &> (F_{xyz}+G_{xyz}) (F_{\overline{x}y\overline{z}}+G_{\overline{x}y\overline{z}}),  \\ 
(F_{x\overline{y}z}+G_{x\overline{y}z}) (F_{xy\overline{z}}+G_{xy\overline{z}}) &> (F_{xyz}+G_{xyz}) (F_{x\overline{y}\overline{z}}+G_{x\overline{y}\overline{z}}). 
\end{align*}
\end{lemma}

\begin{proof}
If all nine inequalities hold, then by Lemma \ref{basic} we get the three logical conclusions
\begin{multline} \label{M:1} F_{\overline{x}yz} G_{xyz}  >  F_{xyz} G_{\overline{x}yz} \ \ \text{and} \ \ F_{x \overline{y}z} G_{xyz}  <  F_{xyz} G_{x \overline{y}z} \quad \text{or} \quad \\
 F_{\overline{x}yz} G_{xyz}  <  F_{xyz} G_{\overline{x}yz} \ \ \text{and} \ \ F_{x \overline{y}z} G_{xyz} > F_{xyz} G_{x \overline{y}z},\end{multline}
\begin{multline}  \label{M:2} F_{\overline{x}yz} G_{xyz}  >  F_{xyz} G_{\overline{x}yz} \ \ \text{and} \ \ F_{xy \overline{z}} G_{xyz}  <  F_{xyz} G_{xy \overline{z}} \quad \text{or} \quad \\
 F_{\overline{x}yz} G_{xyz}  <  F_{xyz} G_{\overline{x}yz} \ \ \text{and} \ \ F_{xy \overline{z}} G_{xyz} > F_{xyz} G_{xy \overline{z}}, \end{multline} 
\begin{multline}  \label{M:3} F_{x\overline{y}z} G_{xyz}  >  F_{xyz} G_{x \overline{y}z} \ \ \text{and} \ \ F_{xy \overline{z}} G_{xyz}  <  F_{xyz} G_{xy \overline{z}} \text{ or } \\ F_{x \overline{y}z} G_{xyz}  <  F_{xyz} G_{x \overline{y}z} \ \ \text{and} \ \ F_{xy \overline{z}} G_{xyz} > F_{xyz} G_{xy \overline{z}}.\end{multline}  
Note that if the first conjunction in \eqref{M:1} is true, then the first conjunction of \eqref{M:3} cannot hold, so the second conjunction of \eqref{M:3} must be true. This in turn implies that the second conjunction of \eqref{M:2} is true, which is a direct contradiction with the first conjunction of \eqref{M:1}.  We get a similar contradiction if we assume the second conjunction of \eqref{M:1} to be true.
\end{proof}

From Lemma \ref{advanced} we can directly draw the following conclusions.

 \begin{corollary}
 There are no Simpson conversions from a triangulation of Type I or II to a triangulation of Type IV.  \qed
 \end{corollary} 
 
  \begin{corollary}
 There are no Simpson conversion from a triangulation of Type VI to a triangulation of Type I or II. \qed
 \end{corollary} 

There is a parity argument in the proof of Lemma \ref{advanced} that relies on the assumption that there are three face diagonals emanating from the vertex $xyz$. This can be extended to the case where there is exactly one face diagonal emanating from $xyz$.  In particular, if a table $F$ has exactly one diagonal incident with vertex $xyz$ and another table $G$ also has only that same face diagonal incident with vertex $xyz$, then the sum $F+G$ cannot have all three face diagonals on the sides incident with $xyz$ different from those in $F$ and $G$. This is the content of the next lemma.

\begin{lemma}\label{parity}
Let $F,G : \{0,1\}^3  \to \R_{+}$ and $(x,y,z) \in \{0,1\}^3 $.  If each of the following six inequalities hold \begin{align*} 
F_{\overline{x}yz} F_{x\overline{y}z} < F_{xyz} F_{\overline{x}\overline{y}z}, 
&\quad 
G_{\overline{x}yz} G_{x\overline{y}z} < G_{xyz} G_{\overline{x}\overline{y}z},  \\ 
F_{\overline{x}yz} F_{xy\overline{z}} > F_{xyz} F_{\overline{x}y\overline{z}}, 
&\quad 
G_{\overline{x}yz} G_{xy\overline{z}} > G_{xyz} G_{\overline{x}y\overline{z}}, \\ 
F_{x\overline{y}z} F_{xy\overline{z}} > F_{xyz} F_{x\overline{y}\overline{z}}, 
&\quad
G_{x\overline{y}z} G_{xy\overline{z}} > G_{xyz} G_{x\overline{y}\overline{z}},  \end{align*} 
then it is not possible for all three of the following three inequalities to hold
\begin{align*} 
(F_{\overline{x}yz}+G_{\overline{x}yz}) (F_{x\overline{y}z}+G_{x\overline{y}z}) &> (F_{xyz}+G_{xyz}) (F_{\overline{x}\overline{y}z}+G_{\overline{x}\overline{y}z}), \\
(F_{\overline{x}yz}+G_{\overline{x}yz}) (F_{xy\overline{z}}+G_{xy\overline{z}}) &< (F_{xyz}+G_{xyz}) (F_{\overline{x}y\overline{z}}+G_{\overline{x}y\overline{z}}),  \\ 
(F_{x\overline{y}z}+G_{x\overline{y}z}) (F_{xy\overline{z}}+G_{xy\overline{z}}) &< (F_{xyz}+G_{xyz}) (F_{x\overline{y}\overline{z}}+G_{x\overline{y}\overline{z}}). 
\end{align*}
\end{lemma}

\begin{proof}
Very similar to proof of Lemma \ref{advanced}.
\end{proof}

\subsection{Infeasible $3$-Dimensional Simpson Conversions}

We are now ready to describe the pairs of triangulations, $A$ and $B$, for which there are no Simpson conversions from $A$ to $B$.  For brevity, we will only list examples of such pairs up to symmetry.  There are $5476$ ordered pairs of triangulations, but only $167$ up to symmetry.\footnote{We calculated the number of distinct ordered pairs of triangulations of the cube up to symmetry by generating the list of all $74^2 = 5476$ ordered pairs and partitioning it into equivalence classes with respect to the relation that two pairs, $(A,B)$ and $(A',B')$, of triangulations of the cube are equivalent if there exists a symmetry of the cube that simultaneously maps $A$ to $A'$ and $B$ to $B'$.}  Among the $167$ symmetry classes of ordered pairs, $55$ satisfy the hypotheses of Lemmas \ref{advanced} and \ref {parity} and therefore cannot give rise to Simpson conversions.  Representatives from each of those symmetry classes of ordered pairs are listed in Table~\ref{T:diagonal1}.

\begin{table}[ht]
\caption{Representatives of the symmetry classes of pairs of triangulations, $A$ and $B$, for which there are no Simpson conversions from $A$ to $B$.}
\begin{tabular}{|c|c|}
\hline Triangulation A & Infeasible Triangulations B \\\hline
1 & 2, 7, 20, 35, 50 \\\hline
3 & 2, 7, 8, 20, 21, 23, 35, 36, 41, 50, 52, 62 \\\hline
11 & 2, 7, 8, 9, 22, 23, 24, 25, 30, 40, 41, 42, 43, 54, 55, 61, 62, 63, 64 \\\hline
35 &  \\\hline
47 & 2, 7, 8, 10, 14, 23, 27, 33, 41, 46, 55, 62, 70 \\\hline
71 & 1, 2, 4, 14, 41, 55 \\\hline
\end{tabular}
 \label{T:diagonal1}
 \end{table}
 
\begin{example} To see how the content of Table~\ref{T:diagonal1} can be used in a practical setting, suppose we measure the presence of a certain mutation at three places in a genome from several individuals and record the data.  We can determine which triangulation the data correspond to by computing the linear forms in Appendix~A. If we find that $e<0$, $f<0$, $j<0$, and $l<0$, then the association corresponds to Triangulation~35 in Appendix~B, which means we can conclude (directly from Table \ref{T:diagonal1}) that the set of individuals cannot be subdivided into two sets, both of which have association corresponding to Triangulation~1 or Triangulation~3. To get the full set of infeasible subdivisions we would also have to look at Triangulations~36 to 46, which are all symmetric to 35. For instance, because Triangulation~43 is a 90 degree rotation about the vertical axis of Triangulation 35, the set of individuals cannot be subdivided into two sets, each of which both have association corresponding to Triangulation~34, which is a 90 degree rotation (in the opposite direction) of Triangulation 11.

The easiest way to conduct this type of inference, however, is to identify empty vertices and apply Lemma~\ref{advanced}.  For instance, the empty vertex at ${111}$ in Triangulation 50 implies that the corresponding table cannot be the sum of any two subtables corresponding to a triangulation with a full vertex at ${111}$ (and similarly for vertex $000$).  One may also apply Lemma~\ref{parity} in a similar manner.
\end{example}

\subsection{Feasible $3$-Dimensional Simpson Conversions}\label{feasible}

We searched for explicit instances of Simpson conversions for each of the remaining $112$ symmetry classes by sampling $2 \times 2 \times 2$ contingency tables uniformly from the probability simplex in $\mathbb{R}^8$ using a method proposed in \cite[Section~2]{Pavlides-Perlman}.  For each representative pair of triangulations, $A$ and $B$, we randomly generated $2 \times 2 \times 2$ contingency tables until we found two that induce Triangulation $A$:  If the sum of those tables induced Triangulation $B$, we recorded the instance of Simpson conversion; otherwise, we repeated the process.  Following that approach, we found explicit instances for all $112$ cases.  The feasible Simpson conversions (up to symmetry) are listed in Table~\ref{T:diagonal2}.

\begin{table}[ht]
\caption{Representatives of the symmetry classes of pairs of triangulations, $A$ and $B$, for which there exists a Simpson conversion from $A$ to $B$.}
\begin{tabular}{|c|C{10.8cm}|}
\hline Triangulation A & Feasible Triangulations B \\\hline
1 & 1, 3, 11, 47, 71 \\\hline
3 & 1, 3, 4, 5, 9, 11, 12, 14, 47, 48, 49, 71, 72 \\\hline
11 & 1, 3, 4, 5, 11, 12, 13, 14, 15, 16, 19, 20, 21, 35, 36, 37, 47, 48, 49, 50, 52, 56, 57, 58, 71, 72, 73 \\\hline
35 & 1, 3, 4, 6, 11, 12, 14, 17, 18, 28, 29, 35, 36, 38, 39, 41, 44, 46, 47, 48, 51, 53, 54, 65, 67, 71, 72, 74 \\\hline
47 & 1, 3, 4, 6, 11, 12, 17, 18, 20, 21, 26, 28, 29, 32, 35, 36, 38, 39, 44, 47, 48, 50, 51, 52, 53, 59, 60, 65, 67, 68, 71, 72, 74 \\\hline
71 & 11, 12, 35, 36, 47, 48, 71, 72 \\\hline
\end{tabular}
 \label{T:diagonal2}
 \end{table}
 
Remarkably, Lemmas \ref{advanced} and \ref{parity} characterise all the cases for which no Simpson conversions can occur, as can be seen in Tables \ref{T:diagonal1} and \ref{T:diagonal2}.  We have thus established the following result.
 
 \begin{theorem}\label{thm:main}
Let $A$ and $B$ be triangulations of the cube.  There exists a Simpson conversion from $A$ to $B$ if and only if there is no vertex of the cube that is incident to an odd number of face diagonals in $A$ and the opposite set of face diagonals in $B$.   \qed
\end{theorem}

The Python script used to generate the data in Tables \ref{T:diagonal1} and \ref{T:diagonal2} is available at: \begin{center} \texttt{http://www.mattstamps.com/simpson/supplementary.zip} \end{center} 
 
\section{Additional Computations \& Future Work}\label{S:amalga}

We conclude with a conjecture and some observations from additional computational experiments, along with some open-ended problems/questions that merit further exploration. 

\subsection{Frequency of Simpson Conversion}

In addition to there being many possible ways in which Simpson conversions can occur, Simpson conversions appear to occur somewhat frequently.  For the 2-dimensional case, \cite{Pavlides-Perlman} experimentally verified -- and presented a proof by Hadjicostas -- that the probability of a Simpson reversal occurring is $1/60$.  More precisely, they showed that if $$\left[p_{ijk} \ \Big| \ i, j, k = 0,1, \ p_{ijk} \geq 0, \text{ and } \sum p_{ijk} = 1\right]$$ is a random $2 \times 2 \times 2$ table (i.e., pairs of $2 \times 2$ tables) from the uniform distribution on the probability simplex in $\mathbb{R}^8$, then the probability that the $2 \times 2$ subtables $$[p_{ij0} \ | \ i, j = 0, 1] \quad \text{and} \quad [p_{ij1} \ | \ i, j = 0, 1]$$ both exhibit positive (or negative) associations while their sum $$[p_{ij0}+p_{ij1} \ | \ i, j = 0, 1]$$ exhibits a negative (resp. positive) association is $1/60$.

To estimate the analogous probability for the $3$-dimensional case, we sampled $2 \times 2 \times 2 \times 2$ tables (i.e., pairs of $2 \times 2 \times 2$ tables) of the form $$\left[p_{ijk\ell} \ \Big| \ i, j, k, \ell = 0, 1, \ p_{ijk\ell} \geq 0, \text{ and } \sum p_{ijk\ell} = 1\right]$$ uniformly from the probability simplex in $\mathbb{R}^{16}$ and calculated the proportion of tables for which the $2 \times 2 \times 2$ subtables $$[p_{ijk0} \ | \ i, j, k = 0,1] \quad \text{and} \quad [p_{ijk1} \ | \ i, j, k = 0,1] $$ induce the same triangulation of the cube (or, equivalently, satisfy identical sets of mutual, marginal, and conditional associations) while their sum $$[p_{ijk0}+p_{ijk1} \ | \ i, j, k = 0,1]$$ induces a different triangulation of the cube.

\begin{conjecture}
The probability that a Simpson conversion occurs in the context above is $1/450$.  
\end{conjecture}

For this estimation, we used the same Dirichlet distribution method as \cite[Section~2]{Pavlides-Perlman} for uniformly sampling the probability simplex in $\mathbb{R}^n$, namely by generating the entries of each table independently according to the $\text{Gamma}(1,1)$ distribution and normalizing.  We generated five million $2 \times 2 \times 2 \times 2$ contingency tables and recorded the percentage of tables that decomposed into two $2 \times 2 \times 2$ contingency tables that induced a common triangulation of the cube, the percentage of those tables that exhibited Simpson conversions, and the percentage of those tables that did not. We repeated this experiment one hundred times, which resulted in $95\%$ confidence intervals that $1.888 \pm 0.012\%$, $0.223 \pm 0.004\%$, and $1.664 \pm 0.012\%$, respectively, which led us to conjecture that the exact values are $17/900$, $2/900$, and $15/900$.  The Python script for the experiment is available at the website listed in Section~\ref{feasible}.

\subsection{Generalized Simpson Conversions}

We also considered the following, more general, version of the main question from Section~\ref{S:switch}:  For which triangulations $A$, $B$, and $C$ with $A \neq B$ is it possible for the sum of a contingency table that induces Triangulation A and a contingency table that induces Triangulation $B$ to induce Triangulation $C$?  Just as before, it is not difficult to find examples of such triples (for instance, in Table~\ref{all_data}, the contingency tables for the Northern, Southern, and all states induce Triangulations 19, 30, and 35, respectively), but Lemmas \ref{advanced} and \ref{parity} still imply that it is not possible to find such examples for every triple of triangulations.  The $199874$ triples of triangulations can be partitioned into $4655$ equivalence classes based on the symmetries of the cube.\footnote{We calculated the number of distinct triples of triangulations of the cube up to symmetry by generating the list of all $\binom{74}{2} \cdot 74 = 199874$ triples and partitioning it into equivalence classes with respect to the relation that a pair of triples, $(A,B,C)$ and $(A',B',C')$, of triangulations are equivalent if there exists a symmetry of the cube that simultaneously maps $A$ to $A'$, $B$ to $B'$, and $C$ to $C'$.}   Of those $4655$ symmetry classes, $351$ cannot occur because of Lemmas \ref{advanced} and \ref{parity}.  Using the same random search technique as the one described in Section~\ref{feasible}, we have found specific instances for 4287 the remaining 4304 cases, leaving 17 unaccounted triples.  Since it is not clear whether these remaining triples are infeasible (for reasons other than Lemmas \ref{advanced} and \ref{parity}) or simply very rare, we are continuing to search for specific instances and maintaining an up-to-date spreadsheet of known instances with the supplementary documents at: \begin{center} \texttt{http://www.mattstamps.com/simpson/feasible-triples.csv} \end{center}

\subsection{Generalization to Higher Dimensions and Further Exploration}

The geometric approach for studying associations among 3-way contingency tables presented in this paper raises a number of questions that merit further exploration.  For instance, what do the 74 non-convex regions in the 8-dimensional probability simplex that correspond to the triangulations of the cube (described in Section~\ref{S:corr}) look like?  How does their geometry relate to that of the regions cut out by the independence hypersurfaces corresponding to the various 3-way (conditional, marginal, and mutual) associations?  Are there formulas for the volumes of these regions that could shed light on the probability of a particular Simpson conversion occurring?  We thank the anonymous referees for suggesting that we state these questions explicitly.         

There is also the question of how the methods in this paper might be extended to higher dimensions.  
While a straightforward modification of the technique described in Sections~\ref{S:2D} and \ref{S:3D} allows one to map a triangulation of the $k$-dimensional cube to each binary $k$-way contingency table, it is not clear what correspondence should exist between multi-way associations and triangulations of high-dimensional cubes.  One could investigate the geometry of the log-linear hypersurfaces in $\mathbb{R}^{2^k}$ that separate the regions corresponding to different triangulations of the $k$-dimensional cube induced by binary $k$-way contingency tables independently of statistical implications, but the triangulations of high-dimensional cubes are not well understood in general.  For instance, the number of different triangulations of a cube grows very rapidly with respect to dimension, even though we only need to consider regular triangulations. It would thus be difficult to produce complete lists of feasible Simpson conversions in dimensions higher than 3. It could, however, be interesting to find generalizations of Lemmas~\ref{advanced} and \ref{parity}.

\section*{Acknowledgements} 
We are grateful to Lior Pachter for asking this question and Soo Go for several helpful suggestions that increased the efficiency and speed of our computations.  We would also like to thank the anonymous referees for the many comments and suggestions that helped us improve the overall quality of this article. This work was supported in part by grant 621-2014-4780 from the Swedish Science Council and National Science Foundation Grant \#1159206.

\bibliographystyle{apalike}
\bibliography{references}

\appendix

\section*{Appendix A: Linear Forms}

The following list of linear forms, taken from \cite{BPS}, is used to determine the triangulation of the cube induced by each $2 \times 2 \times 2$ contingency table.
\begin{align*}
a:=&f_{000}+f_{110}-f_{010}-f_{100}\\
b:=&f_{001}+f_{111}-f_{011}-f_{101}\\
c:=&f_{000}+f_{101}-f_{001}-f_{100}\\
d:=&f_{010}+f_{111}-f_{110}-f_{011}\\
e:=&f_{000}+f_{011}-f_{010}-f_{001}\\
f:=&f_{100}+f_{111}-f_{101}-f_{110}
\end{align*}
\begin{align*}
g:=&f_{000}+f_{111}-f_{011}-f_{100}\\
h:=&f_{001}+f_{110}-f_{010}-f_{101}\\
i:=&f_{000}+f_{111}-f_{010}-f_{101}\\
j:=&f_{001}+f_{110}-f_{011}-f_{100}\\
k:=&f_{000}+f_{111}-f_{001}-f_{110}\\
l:=&f_{010}+f_{101}-f_{011}-f_{100}
\end{align*}
\begin{align*}
m:=&f_{001}+f_{010}+f_{100}-f_{111}-2f_{000}\\
n:=&f_{110}+f_{101}+f_{011}-f_{000}-2f_{111}\\
o:=&f_{100}+f_{010}+f_{111}-f_{001}-2f_{110}\\
p:=&f_{011}+f_{101}+f_{000}-f_{110}-2f_{001}\\
q:=&f_{001}+f_{100}+f_{111}-f_{010}-2f_{101}\\
r:=&f_{110}+f_{011}+f_{000}-f_{101}-2f_{010}\\
s:=&f_{101}+f_{110}+f_{000}-f_{011}-2f_{100}\\
t:=&f_{010}+f_{001}+f_{111}-f_{100}-2f_{011}\\
\end{align*}

\section*{Appendix B: Triangulations}

The following chart comprises the 74 triangulations of the 3-dimensional cube. The numbering is taken from the paper \cite{BPS} for convenience, but we have chosen to list them in a slightly different order to enhance some similarities.\footnote{We discovered two errors in \cite[Table 5.1]{BPS}:  In the row beginning with 57/5, the 43 should be replaced with 42 and, in the row beginning with 63/5, the 42 should be replaced with 43.}  The diagonals on the surface of the cube are blue and the interior diagonals are red.

\section*{Type I}

\begin{center}



\end{center}

\section*{Type II}

\begin{center}

%
\end{center}

\section*{Type III}

\begin{center}

%

\end{center}

\section*{Type IV}

\begin{center}

%

\end{center}

\section*{Type V}

\begin{center}

%

\end{center}

\section*{Type VI}

\begin{center}

%

\end{center}

\end{document}